\documentclass[a4j, 12pt, draft, reqno]{amsart}
\usepackage{amsmath,amssymb}	
\usepackage{amsthm}
\usepackage{mathrsfs}	
\usepackage{bm}
\usepackage{comment}
\usepackage{color}
\usepackage{mathtools}

\mathtoolsset{showonlyrefs=true}

\allowdisplaybreaks[3]

\theoremstyle{plain} 
\newtheorem{theorem}{Theorem}
\newtheorem{lemma}{Lemma}
\newtheorem{corollary}{Corollary}

\newtheorem{proposition}{Proposition}

\newtheorem*{conjecture*}{Conjecture}
\newtheorem*{theorem*}{Theorem}
\newtheorem*{assumption*}{Assumption}

\theoremstyle{plain}

\theoremstyle{remark}

\theoremstyle{definition}

\newtheorem*{notations*}{Notations}
\newtheorem*{acknowledgment*}{Acknowledgments}

\numberwithin{equation}{section}

\let\sl\l
\renewcommand\l{%
	\leavevmode
  \ifmmode
    \left
  \else
    \sl
  \fi}

\newcommand\swapcommand[2]{%
\let\swaptemp#1
\let#1#2
\let#2\swaptemp
}

\newcommand\set[2]{%
\left\{ #1 \; \middle| \; #2 \right\}
}

\swapcommand{\SS}{\S}
\renewcommand{\S}{\mathscr{S}}

\newcommand{\CC}{\mathbb{C}}
\newcommand{\RR}{\mathbb{R}}

\newcommand{\ZZ}{\mathbb{Z}}

\newcommand{\e}{\varepsilon}
\newcommand{\s}{\sigma}

\newcommand{\mc}{\mathcal}

\newcommand{\Lam}{\Lambda}

\newcommand{\qqquad}{\qquad \qquad \qquad}
\newcommand{\qqqquad}{\qquad \qquad \qquad \qquad}

\newcommand{\et}{\tilde{\operatorname{\eta}}}

\newcommand{\meas}{\operatorname{meas}}

\newcommand{\sgn}{\operatorname{sgn}}

\renewcommand{\a}{\alpha}
\renewcommand{\b}{\beta}
\renewcommand{\r}{\right}
\renewcommand{\d}{\displaystyle}
\renewcommand{\Re}{\operatorname{Re}}
\renewcommand{\Im}{\operatorname{Im}}
\renewcommand{\epsilon}{\varepsilon}

\title[Extreme values of $\et_{m}(s)$]
{Extreme values for iterated integrals of the logarithm of the Riemann zeta-function}
\author[S. INOUE]{Sh\={o}ta Inoue}

\address{Graduate School of Mathematics, Nagoya University,
Furocho, Chikusaku, Nagoya 464-8602, Japan}
\email{m16006w@math.nagoya-u.ac.jp}



\begin{document}

\maketitle

\begin{abstract}
In this paper, we give an approximate formula for the measure of extreme values for 
the logarithm of the Riemann zeta-function and its iterated integrals.
The result recovers the unconditional best result for the $\Omega$-result of $S_{1}(t)$ for the part of minus of Tsang.
\end{abstract}

\section{\textbf{Introduction and the statement of results}}

In this paper, we discuss the value distribution of $\et_{m}(s)$ in the critical strip.
Here, the function $\et_{m}(s)$ is defined by the recurrence equation
\begin{align*}
\et_{m}(\s + it) = \int_{\s}^{\infty}\et_{m-1}(\a + it)d\a,
\end{align*}
where $\et_{0}(\s+it) = \log{\zeta(\s + it)}$, and $\zeta(s)$ is the Riemann zeta-function.
Here, we decide the branch of the logarithm of the Riemann zeta-function as follows.
Let $s = \s+it \in \CC$ with $\s, t \in \RR$.
When $t$ is equal to neither zero nor the ordinate of nontrivial zeros of $\zeta(s)$, 
we choose the branch by the continuation with the initial condition
$\lim_{\s \rightarrow +\infty}\log{\zeta(\s + it)} = 0$.
If $t = 0$, then $\log{\zeta(\s)} = \lim_{\e \downarrow 0}\log{\zeta(\s + i\e)}$. 
If $t$ is the ordinate of a nontrivial zero $\rho = \b + i\gamma$ of the Riemann zeta-function, then 
$\log{\zeta(\s + i\gamma)} = \lim_{\e \downarrow 0}\log{\zeta(\s + i(\gamma - \sgn(\gamma)\e))}$.

The function $\et_{m}$ is related to the well known function $S_{m}(t)$, 
which has been studied by many mathematicians including \cite{L1924}, \cite{SS}.
Here, $S_{m}(t)$ is defined by the recurrence equation, for $m \in \ZZ_{\geq 1}$,
\begin{align*}
S_{m}(t) = \int_{0}^{t}S_{m-1}(u)du + b_{m}.
\end{align*}
Here, $S_{0}(t) = \frac{1}{\pi}\Im\log{\zeta(\frac{1}{2}+it)}$, 
and $b_{m} = \frac{1}{\pi (m-1)!}\Im i^{m}\int_{\frac{1}{2}}^{\infty}(\a - \frac{1}{2})^{m-1}\log{\zeta(\a)}d\a$.
To explain the relation between $\et_{m}(\s+it)$ and $S_{m}(t)$, 
we also define the function $\eta_{m}(s)$ by the recurrence equation
\begin{align*}
\eta_{m}(\s + it) = \int_{0}^{t}\eta_{m-1}(\s + iu)du + c_{m}(\s),
\end{align*}
where $\eta_{0}(\s+it) = \log{\zeta(\s + it)}$, and 
$c_{m}(\s) = \frac{i^{m}}{(m-1)!}\int_{\s}^{\infty}(\a - \s)^{m-1}\log{\zeta(\a)}d\a$.
Then the function $S_{m}(t)$ is clearly equal to $\frac{1}{\pi}\Im \eta_{m}(1/2+it)$. 
In the following, $\rho = \b + i\gamma$ means a nontrivial zero of the Riemann zeta-function.
Under this notation, the relation between $\et_{m}(s)$ and $\eta_{m}(s)$ is understood by the equation, for $m \geq 1$,
\begin{align}	\label{RLEE}
&\eta_{m}(\s+it) \\
&= i^{m}\et_{m}(\s+it)
+2\pi \sum_{k = 0}^{m-1}\frac{i^{m-1-k}}{(m-k)! k!}\sum_{\substack{0 < \gamma < t \\ \b > \s}}(\b - \s)^{m-k}(t-\gamma)^{k}.
\end{align}
This can be obtained by Lemma 1 of \cite{II2019} and the fact
$\et_{m}(\s+it) = \frac{1}{(m-1)!}\int_{\s}^{\infty}(\a-\s)^{m-1}\log{(\a+it)}d\a$ 
that can be easily obtained by integration by parts.
Hence, it holds that $S_{1}(t) = \pi^{-1}\Re\et_{1}(1/2+it)$, 
and additionally if the Riemann Hypothesis is true, 
then $S_{m}(t) = \frac{1}{\pi}\Im i^{m}\et_{m}(1/2+it)$ for any $m \in \ZZ_{\geq 0}$.
The difference between $\eta_{m}(1/2+it)$ and $\et_{m}(1/2+it)$ was firstly studied by Fujii \cite{Fu2002}, and
he gave a statement for the magnitude of $\eta_{m}(1/2+it)$ which is equivalent to the Riemann Hypothesis.
From this perspective, the function $\et_{m}(s)$ is an interesting object.

The study of the value distribution of $\et_{m}(s)$ is important because it is directly related to the Lindel\"of Hypothesis.
It is known that (cf. Theorems 13.6 (B) and 13.8 in \cite{T}) 
the Lindel\"of Hypothesis is equivalent to the estimate 
$\Re \et_{1}(1/2+it) = \pi S_{1}(t) = o(\log{t})$ as $t \rightarrow + \infty$.
Additionally, we can generalize this fact to the following proposition.

\begin{proposition}	\label{JGGG}
Let $m \in \ZZ_{\geq 1}$. The Lindel\"of Hypothesis is equivalent to the estimate
$
\Re \et_{m}(1/2+it) = o(\log{t})
$
as $t \rightarrow + \infty$.
\end{proposition}

We omit the proof of this proposition in this paper because
it can be proved by almost the same method as Theorems 13.6 (B) and 13.8 in \cite{T}.
In view of Proposition \ref{JGGG}, it is desirable to understand the exact behavior of $\et_{m}(\s+it)$.
Incidentally, we can show the estimate $\et_{m}(1/2+it) \ll_{m} \log{(|t| + 2)}$ for $m \geq 1$ by the standard way.

Recently, $\Omega$-estimates on $S_{m}(t)$ have been developed by some studies
such as \cite{BS2018}, \cite{Ch2019}, \cite{CM2020} under the Riemann Hypothesis.
Those results were shown by the resonance method due to Bondarenko and Seip \cite{BS2018}, \cite{BS2017}.
On the other hand, as mentioned in \cite{BS2018}, 
it is desired that those could be shown unconditionally by proving a stronger result for the measure of extreme values like 
Soundararajan's result \cite[Theorem 1]{SE2008}.
In this paper, the author shows a result toward this problem.

Now, we define the set $\S_{m, \theta}(T, V; \s)$ by
\begin{align}	\label{def_S}
\S_{m, \theta}(T, V; \s)
:= \set{t \in [T, 2T]}{\Re(e^{-i\theta} \et_{m}(\s + it)) > V}.
\end{align}
The symbol $\meas(\cdot)$ stands for the Lebesgue measure on $\RR$.
Then we show the following theorem.

\begin{theorem}	\label{Main_Thm_EVE}
Let $m \in \ZZ_{\geq 1}$, $\theta \in \RR$ be fixed.
There exists a positive constant $a_{1} = a_{1}(m)$ such that,
for any large numbers $T$, $V$ with $V \leq a_{1}\l(\frac{\log{T}}{(\log{\log{T}})^{2m+2}}\r)^{\frac{m}{2m+1}}$,
we have
\begin{align}	\label{MTEVE1}
\frac{1}{T}\meas(\S_{m, \theta}(T, V; 1/2))
= \exp\l( - 2m 4^{m} V^2(\log{V})^{2m}\l( 1 + R \r) \r), 
\end{align}
where the error term $R$ satisfies
\begin{align*}
R \ll_{m}
\frac{V^{2m+1} (\log{V})^{2m(m+1)}}{(\log{T})^{m}} + \sqrt{\frac{\log{\log{V}}}{\log{V}}}.
\end{align*}
\end{theorem}

This theorem contains the unconditional best result 
$S_{1}(t) = \Omega_{-}\l(\frac{(\log{t})^{1/3}}{(\log{\log{t}})^{4/3}}\r)$ due to Tsang \cite{Ts1986}.
Actually, we can immediately obtain the following corollary.

\begin{corollary}
Let $m \in \ZZ_{\geq 1}$, $\theta \in \RR$ be fixed. Then we have
\begin{align*}
\Re e^{-i\theta} \et_{m}(1/2+it) 
= \Omega_{\pm}\l( \frac{(\log{t})^{\frac{m}{2m+1}}}{(\log{\log{t}})^{\frac{2m^2+2m}{2m+1}}} \r).
\end{align*}
\end{corollary}

If the Riemann Hypothesis is true, we can improve this corollary.
Actually, assuming the Riemann Hypothesis, 
Bondarenko-Seip \cite{BS2018}, Chirre \cite{Ch2019}, and Chirre-Matahab \cite{CM2020} showed that, for certain $\theta$,
\begin{align*}
\Re e^{-i\theta} \et_{m}(1/2+it) 
= \Omega_{+}\l( \frac{\sqrt{\log{t} \log{\log{\log{t}}}}}{(\log{\log{t}})^{m + 1/2}} \r).
\end{align*}
Moreover, by using Tsang's method \cite{Ts1986}, we can prove that for any fixed $\theta \in \RR$, 
\begin{align}	\label{Omega_Tsang}
\Re e^{-i\theta} \et_{m}(1/2+it) 
= \Omega_{+}\l( \frac{\sqrt{\log{t}}}{(\log{\log{t}})^{m + 1/2}} \r)
\end{align}
under the Riemann Hypothesis.
The author cannot find a suitable reference for the latter $\Omega$-result, 
but it is not difficult to check it.
Furthermore, Tsang \cite{Ts1993} showed this $\Omega$-estimate unconditionally in the case $\theta = 0$, $m = 1$.
As we mentioned above, it seems desirable to establish a stronger result for the measure of extreme values 
of $\et_{m}(1/2+it)$ corresponding to the above $\Omega$-results.
Therefore, we hope to prove asymptotic formula \eqref{MTEVE1} for larger $V$, 
but the author was not able to prove it.
In the following, we observe this matter.
To prove Theorem \ref{Main_Thm_EVE}, we in this paper use the fact that 
$\et_{m}(s)$ looks roughly like
\begin{align}	\label{MGHKFCm}
\et_{m}(s) \approx \sum_{p \leq X}\frac{1}{p^{s}(\log{p})^{m}} + \text{contribution from zeros}.
\end{align}
This is an analogue of the fact obtained by the hybrid formula of Gonek, Hughes, and Keating \cite{GHK2007}.
Hence, we need to understand the value distribution of the Dirichlet polynomial $\sum_{p \leq X}\frac{1}{p^{s}(\log{p})^{m}}$
and to estimate the contribution from zeros.
For the value distribution of the Dirichlet polynomial, we can obtain the following proposition.


\begin{proposition}	\label{Main_Prop_EVE}
Let $m \in \ZZ_{\geq 1}$, $\theta \in \RR$ be fixed.
There exist positive constants $a_{2} = a_{2}(m)$, $a_{3} = a_{3}(m)$ such that
for large numbers $T, V, X$ with $V \leq a_{2}\frac{\sqrt{\log{T}}}{(\log{\log{T}})^{m + 1/2}}$, and
$V^{4} \leq X \leq T^{a_{3}/V^2(\log{V})^{2m}}$, we have
\begin{align*}
&\frac{1}{T}\meas\set{t \in [T, 2T]}{\Re e^{-i\theta} \sum_{p \leq X}\frac{1}{p^{1/2+it}(\log{p})^{m}} > V}\\
&=\exp\l( - \frac{2m 4^{m} V^2(\log{V})^{2m}}{1 - \l(\frac{\log{V^2}}{\log{X}}\r)^{m}}
\l( 1 + O_{m}\l( \sqrt{\frac{\log{\log{V}}}{\log{V}}} \r) \r) \r).
\end{align*}
\end{proposition}

This proposition contains an $\Omega$-result for Dirichlet polynomials corresponding to $\Omega$-result \eqref{Omega_Tsang}.
Actually, it follows from Proposition \ref{Main_Prop_EVE} that, for $X = (\log{T})^{4}$,
\begin{align*}
\max_{t \in [T, 2T]}\Re e^{-i\theta} \sum_{p \leq X}\frac{1}{p^{1/2+it}(\log{p})^{m}}
\geq c \frac{\sqrt{\log{T}}}{(\log{\log{T}})^{m + 1/2}}
\end{align*}
for some constant $c > 0$.
Hence, for Dirichlet polynomials, we can prove unconditionally a result for the measure 
that contains the $\Omega$-result corresponding to \eqref{Omega_Tsang}.
On the other hand, we use the result of the previous work of the author \cite[Theorem 5]{II2019} to estimate the contribution 
from zeros based on the zero density estimate of Selberg \cite[Theorem 1]{SCR}. 
However, it is difficult to obtain the satisfactory estimate of the contribution from zeros 
since there are many zeros near the critical line, 
and so the author has been not yet able to prove Theorem \ref{Main_Thm_EVE} for larger $V$.
Additionally, he was not able to prove it even under the Riemann Hypothesis.


So far, we described the results in the case $\s = 1/2$.
On the other hand, the method of the proof of the above assertions can be also applied to the case $\frac{1}{2} < \s < 1$.
Moreover, we can estimate the contribution from zeros in this case satisfactorily.
Thanks to that, we can obtain a theorem 
which is an analogue of the works due to  
Lamzouri \cite{L2011}.
We define $A_{m}(\s)$ by
\begin{align}	\label{def_A_m}
A_{m}(\s)
= \l(\frac{\s^{2\s}}{(1 - \s)^{2\s - 1 + m} G(\s)^{\s}}\r)^{\frac{1}{1-\s}}.
\end{align}
Here, $G(\s) = \int_{0}^{\infty}\log{I_{0}(u)}u^{-1-\frac{1}{\s}}du$, 
and $I_{0}$ is the modified $0$-th Bessel function defined by 
$I_{0}(z) 
= \frac{1}{2\pi}\int_{-\pi}^{\pi}\exp(z \cos{\theta})d\theta 
= \sum_{n = 0}^{\infty}(z / 2)^{2n} / (n!)^2$.

\begin{theorem}	\label{Main_Thm_EVZE}
Let $m \in \ZZ_{\geq 0}$, $\frac{1}{2} < \s < 1$, and $\theta \in \RR$ be fixed.
There exists a positive constant $a_{4} = a_{4}(\s, m)$ such that, 
for any large numbers $T$, $V$ with $V \leq a_{4}\frac{(\log{T})^{1-\s}}{(\log{\log{T}})^{m+1}}$, we have
\begin{align}	\label{MTEVZE1}
\frac{1}{T}\meas(\S_{m, \theta}(T, V; \s))
= \exp\l( - A_{m}(\s) V^{\frac{1}{1-\s}}(\log{V})^{\frac{m+\s}{1 - \s}}\l( 1 + R \r) \r), 
\end{align}
where the error term $R$ satisfies the estimate
\begin{align}	\label{EREEVZE}
R \ll_{\s, m} \sqrt{\frac{1 + m\log{\log{V}}}{\log{V}}}.
\end{align}
\end{theorem}

When $m = 0$, the asymptotic formula of this type was firstly proved by Hattori and Matsumoto\footnote{
There is a difference of the range of $t$ between ours and theirs, but it seems not essential.
Precisely, our range of $t$ is $t \in [T, 2T]$, and theirs is $t \in [-T, T]$.} \cite{HM1999}.
They showed that, for $\frac{1}{2} < \s < 1$,
\begin{align}	\label{HMAF}
&\lim_{T \rightarrow +\infty}\frac{1}{T}\meas\l( \bigcup_{j = 0}^{3}\S_{0, \frac{\pi}{2}j}(T, V; \s) \r)
= \exp\l( -A_{0}(\s)V^{\frac{1}{1-\s}}(\log{V})^{\frac{\s}{1-\s}}(1 + o(1)) \r)
\end{align}
as $V \rightarrow +\infty$.  
Note that the parameter $V$ in their asymptotic formula is not effective with respect to $T$.
Theorem \ref{Main_Thm_EVZE} can recover this asymptotic formula effectively.
Actually, we see that
\begin{align*}
\frac{1}{T}\meas\l( \S_{0, 0}(T, V; \s) \r)
&\leq \frac{1}{T}\meas\l( \bigcup_{j = 0}^{3}\S_{0, \frac{\pi}{2}j}(T, V; \s) \r)\\
&\leq \frac{1}{T}\sum_{j = 0}^{3}\meas\l( \S_{0, \frac{\pi}{2}j}(T, V; \s)\r),
\end{align*}
and both sides are equal to $\exp\l( -A_{0}(\s)V^{\frac{1}{1-\s}}(\log{V})^{\frac{\s}{1-\s}}(1 + R) \r)$ 
from Theorem \ref{Main_Thm_EVZE}. 
Here, the error term $R$ satisfies \eqref{EREEVZE}.
Hence, we can improve \eqref{HMAF} to the effective form. 
On the other hand, it seems this improvement has been essentially obtained by Lamzouri's work \cite{L2011}.
After the study of Hattori-Matsumoto, Lamzouri \cite{L2011} showed an effective asymptotic formula in the case $\theta = 0$ only.
Though he did not mention, we can also prove his theorem for any $\theta \in \RR$ by just using his method. 
Therefore, we may say that the above improvement has been already given by Lamzouri.

Now, we state the proposition corresponding to Proposition \ref{Main_Prop_EVE}, 
which plays an important role in Theorem \ref{Main_Thm_EVZE}.

\begin{proposition}	\label{Main_Prop_EVZE}
Let $m \in \ZZ_{\geq 0}$, $\frac{1}{2} < \s < 1$, and $\theta \in \RR$ be fixed. 
There exist positive constants $a_{5} = a_{5}(\s, m)$, $a_{6} = a_{6}(\s, m)$ such that
for large numbers $T, X, V$ with $V \leq a_{5}\frac{(\log{T})^{1-\s}}{(\log{\log{T}})^{m+1}}$ and
$V^{\frac{4\s}{1-\s}} \leq X \leq T^{a_{6} /  V^{\frac{1}{1-\s}} (\log{V})^{\frac{m+\s}{1-\s}}}$,
we have
\begin{align*}
&\frac{1}{T}\meas\set{t \in [T, 2T]}{\Re\bigg( e^{-i\theta} \sum_{p \leq X}\frac{1}{p^{\s+it}(\log{p})^{m}}\bigg) > V}\\
&= \exp\l( - A_{m}(\s) V^{\frac{1}{1-\s}}(\log{V})^{\frac{m+\s}{1 - \s}}
\l( 1 + O_{\s, m}\l(\sqrt{\frac{1 + m\log{\log{V}}}{\log{V}}}\r)\r) \r).
\end{align*}
\end{proposition}

As we mentioned above, we can obtain a good estimate of the contribution from zeros, 
and so Theorem \ref{Main_Thm_EVZE} is proved in the same range as Proposition \ref{Main_Prop_EVZE}.

Here, we describe the method of the proofs of Theorem \ref{Main_Thm_EVE} and Theorem \ref{Main_Thm_EVZE} roughly.
These theorems are analogues of Lamzouri's result, but we cannot adopt directly his method.
He used the Euler product of the Riemann zeta-function and the generalized divisor function to estimate a Dirichlet polynomial. 
However, $\et_{m}(s)$ does not have the representation of Euler product when $m \geq 1$, 
and so we cannot apply directly his method. 
To avoid this obstacle the author uses \eqref{MGHKFCm}, and estimates the Dirichlet polynomial by using Radziwi\l\l's 
method \cite{Ra2011}.

\section{\textbf{Preliminaries}}

In this section, we prepare some lemmas.

\begin{lemma}	\label{RCLG}
Let $\theta \in \RR$ be fixed. 
For any $n \in \ZZ_{\geq 2}$, we write $n = q_1^{\a_1} \dots q_{r}^{\a_r}$, where $q_{j}$ are distinct prime numbers.
Then we have
\begin{align}
\frac{1}{T}\int_{T}^{2T}\prod_{j = 1}^{r}\l(\cos(t\log{q_{j}} + \theta)\r)^{\a_j}dt
= f(n) + O\l( \frac{n}{T} \r)
\end{align}
for any $T > 0$. Here, $f$ is the multiplicative function defined by
$
f(p^{\a}) = 
2^{-\a}\begin{pmatrix}
\a\\
\a/2
\end{pmatrix}
$
for a prime power $p^{\a}$,
and we regard that 
$\begin{pmatrix}
\a\\
\a/2
\end{pmatrix} = 0$ 
if $\a$ is odd.
\end{lemma}

\begin{proof}
We find that
\begin{align*}
&(\cos(t\log{q_{j}} + \theta))^{\a_{j}}\\
&= \frac{1}{2^{\a_{j}}}\l( e^{i(t\log{q_{j}} + \theta)} + e^{-i(t\log{q_{j}} + \theta)} \r)^{\a_{j}}
= \frac{1}{2^{\a_{j}}}\sum_{\e_{1}, \dots, \e_{\a_{j}} \in \{-1, 1 \}}e^{i(\e_1 + \cdots + \e_{\a_{j}})(t \log{q_{j}} + \theta)}\\
&= \frac{1}{2^{\a_{j}}}
\begin{pmatrix}
\a_{j}\\
\a_{j}/2
\end{pmatrix}
+\frac{1}{2^{\a_{j}}}\sum_{\substack{\e_{1}, \dots, \e_{\a_{j}} \in \{ -1, 1 \} \\ \e_{1} + \cdots + \e_{\a_{j}} \not= 0}}
e^{i(\e_1 + \cdots + \e_{\a_{j}})(t \log{q_{j}} + \theta)}.
\end{align*}
Therefore, we obtain
\begin{align}	\label{RCLG1}
\prod_{j = 1}^{r}\l(\cos(t\log{q_{j} + \theta})\r)^{\a_{j}}
= f(n) + E,
\end{align}
where $E$ is the sum whose the number of terms is less than $2^{\Omega(n)}$, 
and the form of each term is $\delta e^{it(\b_1\log{q_{1}} + \cdots + \b_{r}\log{q_{r}})}$. 
Here, $\delta$ is a complex number independent of $t$ satisfying $|\delta| \leq 2^{-\Omega(n)}$, 
and $\b_{j}$'s are integers with $0 \leq |\b_{j}| \leq \a_j$ and $\b_{s} \not= 0$ for some $1 \leq s \leq r$.
Since $|\b_{1}\log{q_{1}} + \cdots \b_{r}\log{q_{r}}| \gg n^{-1}$, 
the integral of each term of $E$ is estimated by $\ll n 2^{-\Omega(n)}$. 
As the number of such terms $\ll 2^{\Omega(n)}$, we have $\int_{T}^{2T}E dt \ll n$. 
Thus, by this estimate and equation \eqref{RCLG1}, we obtain this lemma.
\end{proof}

\begin{lemma}	\label{RKL}
Let $m \in \ZZ_{\geq 0}$, $\frac{1}{2} \leq \s < 1$ be fixed.
Let $X \geq 3$, and $T$ be large. 
Then, for any positive integer $k$, we have
\begin{align*}
&\frac{1}{T}\int_{T}^{2T}\l( \Re\l(e^{-i\theta}\sum_{p \leq X}\frac{1}{p^{\s+it}(\log{p})^{m}} \r) \r)^{k}dt\\
&= \frac{k!}{2\pi i}\oint_{|w| = R}\frac{1}{w^{k+1}}\prod_{p \leq X}I_{0}\l( \frac{w}{p^{\s}(\log{p})^{m}} \r)dw
+O\l(\frac{X^{2k}}{T} \r).
\end{align*}
Here, $R$ is any positive number, 
and $I_{0}$ is the modified $0$-th order Bessel function.
\end{lemma}

\begin{proof}
Define the multiplicative function $g_{X}(n)$ as, for every prime number $p$ and $\a \in \ZZ_{\geq 1}$, 
$g_{X}(p^{\a}) = 1 / \a!(\log{p})^{\a m}$ if $p \leq X$, and $g_{X}(p^{\a}) = 0$ otherwise.
By Lemma \ref{RCLG}, we find that
\begin{align*}
&\frac{1}{T}\int_{T}^{2T}\l( \Re\l(e^{-i\theta}\sum_{p \leq X}\frac{1}{p^{\s+it}(\log{p})^{m}} \r) \r)^{k}dt\\
&= \frac{1}{T}\sum_{p_1, \dots, p_k \leq X}\frac{\int_{T}^{2T}\cos(t \log{p_1}+\theta) \cdots \cos(t \log{p_k}+\theta)dt}
{(p_1 \cdots p_{k})^{\s} (\log{p_1} \cdots \log{p_k})^{m}}\\
&= \sum_{p_1, \dots, p_k \leq X}\frac{f(p_1 \cdots p_k)}{(p_1 \cdots p_k)^{\s}(\log{p_1} \cdots \log{p_k})^{m}} 
+ O\l( \frac{X^{2k}}{T} \r).
\end{align*}
From this equation and the definition of $g_{X}$, 
we have
\begin{align}	\label{RKLE1}
\frac{1}{T}\int_{T}^{2T}\l(\Re\sum_{p \leq X}\frac{1}{p^{\s+it}(\log{p})^{m}} \r)^{k}dt
= k! \sum_{\Omega(n) = k}\frac{f(n)}{n^{\s}}g_{X}(n) + O\l( \frac{X^{2k}}{T} \r).
\end{align}
By Cauchy's integral formula, the above is equal to
\begin{align*}
\frac{k!}{2\pi i}\oint_{|w| = R}\sum_{n = 1}^{\infty}\frac{f(n)}{n^{\s}}g_{X}(n)w^{\Omega(n)}\frac{dw}{w^{k+1}}
+O\l( \frac{X^{2k}}{T} \r).
\end{align*}
Since the functions $f$, $g_{X}$, and $w^{\Omega(n)}$ are multiplicative, this main term is 
\begin{align*}
&= \frac{k!}{2\pi i}\oint_{|w| = R}\frac{1}{w^{k+1}}\prod_{p \leq X}
\l(\sum_{l=0}^{\infty}\l( \frac{(w / 2p^{\s} (\log{p})^{m})^{2l}}{(l!)^2} \r) \r)dw\\
&= \frac{k!}{2\pi i}\oint_{|w| = R}\frac{1}{w^{k+1}}\prod_{p \leq X}I_{0}\l( \frac{w}{p^{\s}(\log{p})^{m}} \r)dw,
\end{align*}
which completes the proof of this lemma.
\end{proof}

\begin{lemma}	\label{PB0E}
Let $m$ be a fixed positive interger.
For $x \geq 3$, $X \geq x^{3}$, we have
\begin{align}	\label{PB0E1}
&\prod_{p \leq X}I_{0}\l( \frac{x}{\sqrt{p}(\log{p})^{m}} \r)\\
&= \exp\l( \frac{x^2}{8m(2\log{x})^{2m}}\l( 1 - \l(\frac{\log{x^2}}{\log{X}}\r)^{2m} 
+ O\l( \frac{\log{\log{x}}}{\log{x}} \r) \r) \r).
\end{align}
\end{lemma}

\begin{proof}
By the Taylor expansion of $I_{0}$ and the prime number theorem, we find that
\begin{align}
&\prod_{\frac{x^{2}}{(\log{x})^{2m}} < p \leq X}I_{0}\l( \frac{x}{\sqrt{p}(\log{p})^{m}} \r)\\
&= \exp\l( \sum_{\frac{x^{2}}{(\log{x})^{2m}} < p \leq X} \l(\frac{x^2}{4p (\log{p})^{2m}} 
+ O_{m}\l( \frac{x^4}{p^2 (\log{p})^{4m}} \r) \r) \r)\\ \label{PB0E2}
&= \exp\l( \frac{x^2}{8m(2\log{x})^{2m}}\l( 1 - \l(\frac{\log{x^2}}{\log{X}}\r)^{2m}
+ O_{m}\l( \frac{\log{\log{x}}}{\log{x}} \r) \r) \r).
\end{align}
On the other hand, by using the inequality $I_{0}(x) \leq \exp(x)$ and the prime number theorem, 
it holds that
\begin{align*}
\prod_{p \leq \frac{x^{2}}{(\log{x})^{2m}}}I_{0}\l( \frac{x}{\sqrt{p}(\log{p})^{m}} \r)
&\leq \exp\l( x \sum_{p \leq \frac{x^2}{(\log{x})^{2m}}}\frac{1}{\sqrt{p}(\log{p})^{m}} \r)\\
&\leq \exp\l( O_{m}\l( \frac{x^2}{(\log{x})^{2m+1}} \r) \r).
\end{align*}
From this estimate and equation \eqref{PB0E2}, we obtain this lemma.
\end{proof}

\begin{lemma}	\label{PB0EG}
Let $\frac{1}{2} < \s < 1$, $m \in \ZZ_{\geq 0}$ be fixed.
Then, for large $x$, $X \geq x^{3}$, we have
\begin{align*}	
\prod_{p \leq X}I_{0}\l( \frac{x}{p^{\s}(\log{p})^{m}} \r)
=  \exp\l(\frac{\s^{\frac{m}{\s}}G(\s)x^{\frac{1}{\s}}}{(\log{x})^{\frac{m}{\s}+1}}
\l( 1 + O\l(\frac{1 + m\log{\log{x}}}{\log{x}} \r) \r) \r).
\end{align*}
\end{lemma}

\begin{proof}
We take the numbers $y_0$, $y_1$ as satisfying the equations 
$y_{0}^{\s}(\log{y_{0}})^{m} = x^{1/2}$, $y_{1}^{\s}(\log{y_{1}})^{m} = x^{3/2}$, respectively.
Then, it holds that $y_{0} \asymp_{m} x^{\frac{1}{2\s}}(\log{x})^{-\frac{m}{\s}}$, 
$y_{1} \asymp_{m} x^{\frac{3}{2\s}}(\log{x})^{-\frac{m}{\s}}$, and the estimate $X \gg y_1$ also holds.
By the Taylor expansion of $I_{0}$ and the prime number theorem, we find that
\begin{align*}
\sum_{p \leq X}\log{I_{0}\l( \frac{x}{p^{\s}(\log{p})^{m}} \r)}
= \sum_{p \leq y_1}\log{I_{0}\l( \frac{x}{p^{\s}(\log{p})^{m}} \r)}
+ O_{m, \s}\l(\frac{x^{\frac{3-2\s}{2\s}}}{(\log{x})^{\frac{m}{\s} + 1}} \r).
\end{align*}
By the inequality $I_{0}(x) \leq \exp(x)$, it holds that
\begin{align*}
\sum_{p \leq y_0}\log{I_{0}\l( \frac{x}{p^{\s}(\log{p})^{m}} \r)}
\leq \sum_{p \leq y_{0}}\frac{x}{p^{\s}(\log{p})^{m}}
\ll_{m, \s} \frac{x^{\frac{1-\s}{2\s}}}{(\log{x})^{\frac{m}{\s} + 1}}.
\end{align*}
From these estimates, one has
\begin{align}	\label{PB0EG4}
&\sum_{p \leq X}\log{I_{0}\l( \frac{x}{p^{\s}(\log{p})^{m}} \r)}\\
&= \sum_{y_0 < p \leq y_{1}}\log{I_{0}\l( \frac{x}{p^{\s}(\log{p})^{m}} \r)}
+O_{m, \s}\l( \frac{1}{(\log{x})^{\frac{m}{\s} + 1}}
\l( x^{\frac{3 - 2\s}{2\s}} + x^{\frac{1+\s}{2\s}} \r) \r).
\end{align}

By using partial summation and estimates of $I_{0}$, we obtain
\begin{align}	\label{PB0EG3}
&\sum_{y_0 < p \leq y_1}\log{I_{0}\l( \frac{x}{p^{\s}(\log{p})^{m}} \r)}\\
&= - \int_{y_{0}+}^{y_{1}+}\pi(\xi)\l(\frac{d}{d\xi}\log{I_{0}\l( \frac{x}{\xi^{\s}(\log{\xi})^{m}} \r)} \r)d\xi
+O_{m}\l( \frac{x^{\frac{1+\s}{2\s}} + x^{\frac{3-2\s}{2\s}}}{(\log{x})^{\frac{m}{\s} + 1}} \r).
\end{align}
Applying the basic formula $\pi(\xi) = \int_{2}^{\xi}\frac{du}{\log{u}} + O(\xi e^{-c\sqrt{\log{\xi}}})$, 
we find that the first term on the right hand side is equal to
\begin{align}	\label{PB0EG2}
\int_{y_0}^{y_1}\frac{\log{I_{0}\l( \frac{x}{\xi^{\s}(\log{\xi})^{m}} \r)}}{\log{\xi}}d\xi
+ O\l( \int_{y_0}^{y_1}e^{-c\sqrt{\log{\xi}}}\log{I_{0}\l( \frac{x}{\xi^{\s}(\log{\xi})^{m}} \r)}d\xi \r).
\end{align}
Note that we used the monotonicity of $I_{0}$ in the above deforming.
By the estimate $I_{0}(x) \leq \exp(x)$ and the Taylor expansion of $I_{0}(z)$, we find that
\begin{align*}
&\int_{y_0}^{y_1}e^{-c\sqrt{\log{\xi}}}\log{I_{0}\l( \frac{x}{\xi^{\s}(\log{\xi})^{m}} \r)}d\xi\\
&\ll_{m} x\int_{y_0}^{\frac{x^{1/\s}}{(\log{x})^{m/\s}}}\frac{d\xi}{\xi^{\s}(\log{\xi})^{2m+3}}
+ x^{2}\int_{\frac{x^{1/\s}}{(\log{x})^{m/\s}}}^{\infty}\frac{d\xi}{\xi^{2\s}(\log{\xi})^{2m+3}}\\&
\ll \frac{x^{\frac{1}{\s}}}{(\log{x})^{\frac{m}{\s}+2}}.
\end{align*}
Finally, we consider the first term of \eqref{PB0EG2}.
By making the change of variables $u = \frac{x}{\xi^{\s}(\log{\xi})^{m}}$, 
hard but not difficult calculations can lead that the first term of \eqref{PB0EG2} is equal to
\begin{align*}
&\s^{m/\s}x^{1/\s}\int_{x^{-1/2}}^{x^{1/2}}\frac{(1 + O_{m}(\frac{m\log{\log{x}}}{\log{x}}))
\log{I_{0}(u)}}{u^{1+\frac{1}{\s}}(\log{(x/u)})^{\frac{m}{\s}+1}}du\\
&= \s^{m/\s}x^{1/\s}\int_{x^{-1/2}}^{x^{1/2}}\frac{\log{I_{0}(u)}}{u^{1+\frac{1}{\s}}(\log{(x/u)})^{\frac{m}{\s}+1}}du
+ O_{m, \s}\l(\frac{mx^{1/\s}\log{\log{x}}}{(\log{x})^{\frac{m}{\s}+2}}\r).
\end{align*}
Since $\frac{1}{(\log{(x/u)})^{m/\s + 1}} = \frac{1 + O_{m}( |\log{u}|/\log{x} )}{(\log{x})^{m/\s + 1}} $
for $x^{-1/2} \leq u \leq x^{1/2}$, we find that
\begin{align*}
&\int_{x^{-1/2}}^{x^{1/2}}\frac{\log{I_{0}(u)}}{u^{1+\frac{1}{\s}}(\log{(x/u)})^{\frac{m}{\s}+1}}du\\
&= \frac{1}{(\log{x})^{\frac{m}{\s}+1}}
\int_{x^{-1/2}}^{x^{1/2}}\frac{\log{I_{0}(u)}}{u^{1+\frac{1}{\s}}}du
+O_{m}\l( \frac{1}{(\log{x})^{\frac{m}{\s}+2}}\int_{x^{-1/2}}^{x^{1/2}}\frac{\log{I_{0}(u)}|\log{u}|}{u^{1 + \frac{1}{\s}}}du \r).
\end{align*}
Moreover, by $I_{0}(x) \leq \exp(x)$ and the Taylor expansion of $I_0$, it holds that
\begin{align*}
\int_{x^{-1/2}}^{x^{1/2}}\frac{\log{I_{0}(u)}}{u^{1+\frac{1}{\s}}}du
= \int_{0}^{\infty}\frac{\log{I_{0}(u)}}{u^{1+\frac{1}{\s}}}du
+O_{\s}\l( x^{\frac{1 - 2\s}{2\s}} + x^{\frac{\s - 1}{2\s}} \r),
\end{align*}
and that
\begin{align*}
\int_{x^{-1/2}}^{x^{1/2}}\frac{\log{I_{0}(u)|\log{u}|}}{u^{1+\frac{1}{\s}}}du
\ll_{\s} 1
\end{align*}
for $\frac{1}{2} < \s < 1$.
From the above calculations,  
equation \eqref{PB0EG3} is
\begin{align*}
= \frac{\s^{\frac{m}{\s}}G(\s)x^{\frac{1}{\s}}}{(\log{x})^{\frac{m}{\s}+1}}
\l( 1 + O\l(\frac{1 + m\log{\log{x}}}{\log{x}}\r) \r).
\end{align*}
Hence, by estimates \eqref{PB0EG4}, \eqref{PB0EG3}, \eqref{PB0EG2}, we obtain this lemma.
\end{proof}

\begin{lemma}	\label{SLL}
Let $T$ be large, and let $3 \leq X \leq T$. Let $k$ be a positive integer such that $X^{k} \leq T / \log{T}$.
For any complex numbers $a(p)$ we have
\begin{align*}
\int_{T}^{2T}\bigg| \sum_{p \leq X}\frac{a(p)}{p^{1/2+it}} \bigg|^{2k}dt
\ll T k! \l( \sum_{p \leq X}\frac{|a(p)|^2}{p} \r)^{k}.
\end{align*}
\end{lemma}

\begin{proof}
This is Lemma 3 in \cite{SM2009}.
\end{proof}

\begin{lemma}	\label{ESAE}
Let $m \in \ZZ_{\geq 0}$, $\frac{1}{2} \leq \s < 1$ be fixed with $(m, \s) \not= (0, 1/2)$.
Let $T$, $W$ be large numbers. 
Put $\kappa(\s) = 0$ if $\s = 1/2$, $\kappa(\s) = \s$ otherwise.
Define the set $\mathcal{A} = \mathcal{A}(T, X, W; \s, m)$ by
\begin{align}	\label{def_sA}
\mathcal{A} = \set{t \in [T, 2T]}{\bigg|\sum_{p \leq X}\frac{1}{p^{\s+it}(\log{p})^{m}} \bigg| \leq W}.
\end{align}
Then, there exists a small positive constant $b_{1} = b_{1}(\s, m) \leq 1$ such that
for any $3 \leq X \leq T^{1 / W^{\frac{1}{1-\s}} (\log{W})^{\frac{m+\kappa(\s)}{1-\s}}}$, 
\begin{align*}
\frac{1}{T}\meas([T, 2T] \setminus \mathcal{A}) \ll 
\exp\l( - b_{1} W^{\frac{1}{1-\s}} (\log{W})^{\frac{m+\kappa(\s)}{1-\s}} \r).
\end{align*}
\end{lemma}

\begin{proof}
Using the prime number theorem, we can obtain
\begin{align*}
\sum_{p \leq k(\log{k})^{2 - \kappa(\s)}}\frac{1}{p^{\s+it}(\log{p})^{m}}
\ll_{m} \frac{k^{1-\s}}{(\log{k})^{m + \kappa(\s)}}.
\end{align*}
By Lemma \ref{SLL}, we have
\begin{align*}
\frac{1}{T}\int_{T}^{2T}\bigg| \sum_{k(\log{k})^{2 - \kappa(\s)} < p \leq X}
\frac{1}{p^{\s+it}(\log{p})^{m}} \bigg|^{2k}dt
&\ll k! \l( \sum_{p > k(\log{k})^{2 - \kappa(\s)}}\frac{1}{p^{2\s} (\log{p})^{2m}} \r)^{k}\\
&\leq \l(C_{1} \frac{k^{1 - \s}}{(\log{k})^{m + \kappa(\s)}} \r)^{2k}
\end{align*}
for $X^{k} \leq T^{1/2}$, where $C_{1} = C_{1}(\s, m)$ is a positive constant.
Therefore, when $X^{k} \leq T^{1/2}$ it holds that
\begin{align}	\label{MVEDPm}
\frac{1}{T}\int_{T}^{2T}\bigg| \sum_{p \leq X}
\frac{1}{p^{\s+it}(\log{p})^{m}} \bigg|^{2k}dt
\leq \l(C_{2} \frac{k^{1 - \s}}{(\log{k})^{m + \kappa(\s)}} \r)^{2k}
\end{align}
for some constant $C_{2} = C_{2}(\s, m) > 0$.
Hence, we have
\begin{align*}
\frac{1}{T}\meas([T, 2T] \setminus \mathcal{A})
\leq \l(C_{2} \frac{k^{1 - \s}}{W (\log{k})^{m + \kappa(\s)}} \r)^{2k}.
\end{align*}
Choosing $k = [c W^{\frac{1}{1-\s}} (\log{W})^{\frac{m+\kappa(\s)}{1-\s}}]$ 
with $c = c(\s, m)$ a suitably small constant, we obtain this lemma.
\end{proof}


\begin{lemma}	\label{GRKL}
Assume the same situation as in Lemma \ref{ESAE}. 
There exists a small positive constant $b_{2} = b_{2}(\s, m)$ such that 
for $3 \leq x \leq b_{2}W^{\frac{\s}{1-\s}} (\log{W})^{\frac{m + \kappa(\s)}{1-\s}}$, 
$x^{3} \leq X \leq T^{1 / W^{\frac{1}{1-\s}} (\log{W})^{\frac{m+\kappa(\s)}{1-\s}}}$, we have
\begin{align}	\label{GRKL2}
&\frac{1}{T}\int_{\mathcal{A}}
\exp\l( x\Re\bigg( e^{-i\theta} \sum_{p \leq X}\frac{1}{p^{\s+it}(\log{p})^{m}}\bigg) \r)dt\\
&= \prod_{p \leq X}I_{0}\l( \frac{x}{p^{\s}(\log{p})^{m}} \r) 
+ O\l( \exp\l( - x W \r) \r).
\end{align}
\end{lemma}

\begin{proof}
By the definition of $\mathcal{A}$ and the Stirling formula, we have
\begin{align}	\label{GRKL3}
&\int_{A}\exp\l( x\Re \l( e^{-i\theta} \sum_{p \leq X}\frac{1}{p^{\s+it}(\log{p})^{m}} \r) \r)dt\\
&= \sum_{k \leq Y}\frac{x^{k}}{k!}\int_{A}
\bigg(\Re \sum_{p \leq X}\frac{e^{-i\theta}}{p^{\s+it}(\log{p})^{m}}\bigg)^{k}dt 
+ O\l(T \sum_{k > Y}\frac{1}{\sqrt{k}}\l( \frac{ex W }{k} \r)^{k} \r),
\end{align}
where $Y = e^{2} x W$.
Here, an easy calculation for geometric sequence shows that the above $O$-term is
$
\ll T\exp\l(- e^{2} x W \r).
$
By using the Cauchy-Schwarz inequality, we find that
\begin{multline*}
\int_{A}\bigg(\Re\sum_{p \leq X}\frac{e^{-i\theta}}{p^{\s+it}(\log{p})^{m}}\bigg)^{k}dt
= \int_{T}^{2T}\bigg(\Re\sum_{p \leq X}\frac{e^{-i\theta}}{p^{\s+it}(\log{p})^{m}}\bigg)^{k}dt +\\
+O\l( (\meas([T, 2T] \setminus \mathcal{A}))^{1/2}
\l(\int_{T}^{2T} \bigg|\sum_{p \leq X}\frac{1}{p^{\s+it}(\log{p})^{m}}\bigg|^{2k}dt\r)^{1/2} \r).
\end{multline*}
When $b_{2} \leq e^{-2}$, from estimate \eqref{MVEDPm} and Lemma \ref{ESAE}, this $O$-term is
\begin{align*}
\ll T\exp\l( -\frac{b_{1}}{2} W^{\frac{1}{1-\s}} (\log{W})^{\frac{m+\kappa(\s)}{1-\s}} \r)
\l( C_{2}\frac{k^{1-\s}}{(\log{k})^{m+\kappa(\s)}} \r)^{k} 
\end{align*}
for $k \leq Y$, where $C_{2} = C_{2}(\s, m)$ is a positive constant.
Also, it holds that
\begin{align*}
\sum_{0 \leq k \leq Y}\frac{x^{k}}{k!}\l( C_{2}\frac{k^{1-\s}}{(\log{k})^{m+\kappa(\s)}} \r)^{k}
&\leq \sum_{k = 0}^{\infty}\frac{1}{k!}\l( C_{2}\frac{x Y^{1-\s}}{(\log{Y})^{m + \kappa(\s)}} \r)^{k}\\
&\leq \exp\l( 2b_{2}^{2-\s} C_{2} W^{\frac{1}{1-\s}}(\log{W})^{\frac{m + \kappa(\s)}{1-\s}} \r)
\end{align*}
for any sufficiently large $W$.
Therefore, choosing $b_{2}$ suitably small, we find that the right hand side is 
$\leq \exp\l( \frac{b_{1}}{6}W^{\frac{1}{1-\s}}(\log{W})^{\frac{m+\kappa(\s)}{1-\s}} \r)$.
Hence, we obtain
\begin{align*}
&\sum_{k \leq Y}\frac{x^{k}}{k!}\int_{A}\bigg(\Re\sum_{p \leq X}\frac{e^{-i\theta}}{p^{\s+it}(\log{p})^{m}}\bigg)^{k}dt\\
&= \sum_{k \leq Y}\frac{x^{k}}{k!}\int_{T}^{2T}
\bigg(\Re\sum_{p \leq X}\frac{e^{-i\theta}}{p^{\s+it}(\log{p})^{m}}\bigg)^{k}dt +\\
&\qqqquad \qqquad 
+ O\l( T\exp\l( -\frac{b_{1}}{3} W^{\frac{1}{1-\s}} (\log{W})^{\frac{m+\kappa(\s)}{1-\s}} \r) \r).
\end{align*}
From these estimates, the left hand side of \eqref{GRKL3} is equal to
\begin{align}	\label{GRKL5}
\sum_{k \leq Y}\frac{x^{k}}{k!}
\int_{T}^{2T}\bigg(\Re\sum_{p \leq X}\frac{e^{-i\theta}}{p^{\s+it}(\log{p})^{m}}\bigg)^{k}dt
+ O\l( T\exp\l( - e^{2} x W \r) \r)
\end{align}
for any sufficiently large $W$ when $b_{2}$ is suitably small.
By Lemma \ref{RKL}, this main term is equal to
\begin{align} \label{GRKL4}
\frac{T}{2\pi i}\oint_{|w| = e x}
\sum_{k \leq Y}\frac{x^{k}}{w^{k+1}}\prod_{p \leq X}I_{0}\l( \frac{w}{p^{\s}(\log{p})^{m}} \r)dw.
\end{align}
By Lemmas \ref{PB0E} and \ref{PB0EG}, there exists a constant $C_{4} = C_{4}(\s, m) > 0$ such that
\begin{align*}
\l|\prod_{p \leq X}I_{0}(w / p^{\s}(\log{p})^{m})\r|
&\leq I_{0}(R / p^{\s}(\log{p})^{m})
\leq \exp\l(C_{4} \frac{x^\frac{1}{\s}}{(\log{x})^{\frac{m + \kappa(\s)}{\s}}} \r).
\end{align*}
Choosing $b_{2}$ as a suitably small constant, the right hand side is $\ll \exp(x W)$.
Moreover, since we see that
\begin{align*}
\l|\sum_{k > Y}\frac{x^{k}}{w^{k+1}}\r|
\ll \exp\l( - e^2 x W \r),
\end{align*}
it holds that
\begin{align*}
&\l|\sum_{k > Y}\frac{x^{k}}{w^{k+1}}\prod_{p \leq X}I_{0}\l( \frac{w}{\sqrt{p}(\log{p})^{m}} \r)\r|
\leq \exp\l( - x W \r)
\end{align*}
for $|w| = e x$. 
Hence, \eqref{GRKL4} is equal to
\begin{align*}
\frac{T}{2\pi i}\oint_{|w| = ex}\sum_{k \leq Y}\frac{1}{w - x}\prod_{p \leq X}I_{0}\l( \frac{w}{p^{\s}(\log{p})^{m}} \r)dw
+ O\l( T\exp\l( - x W \r) \r).
\end{align*}
Thus, by this formula and equation \eqref{GRKL5} and using Cauchy's integral formula, we obtain
\begin{align}
&\frac{1}{T}\int_{A}\exp\l( x\Re\bigg(e^{-i\theta}\sum_{p \leq X}\frac{1}{p^{\s+it}(\log{p})^{m}}\bigg) \r)dt\\
&= \prod_{p \leq X}I_{0}\l( \frac{x}{p^{\s}(\log{p})^{m}} \r) 
+ O\l( \exp\l( - x W \r) \r),
\end{align}
which completes the proof of this lemma.
\end{proof}

\begin{lemma}	\label{ESEPE}
Let $m \in \ZZ_{\geq 1}$, $\frac{1}{2} \leq \s < 1$ be fixed.
Let $T$ be large, $X \geq 3$, and $\Delta > 0$. 
Define the set $\mathcal{B} = \mathcal{B}(T, X, \Delta; \s)$ by
\begin{align*}
\mathcal{B} = \set{t \in [T, 2T]}{\l|\et_{m}(\s + it) 
- \sum_{2 \leq n \leq X}\frac{\Lam(n)}{n^{\s+it}(\log{n})^{m+1}}\r| \leq \Delta X^{1/2-\s}}.
\end{align*}
Then, for $0 < \Delta \leq \l(\frac{\log{T}}{(\log{X})^{2(m+1)}}\r)^{\frac{m}{2m+1}}$, we have
\begin{align*}
\frac{1}{T}\meas([T, 2T] \setminus \mathcal{B}) 
\leq \exp\l( -b_{3} \Delta^2(\log{X})^{2m} \r),
\end{align*}
and for $ \l(\frac{\log{T}}{(\log{X})^{2(m+1)}}\r)^{\frac{m}{2m+1}}
\leq \Delta \leq \frac{\log{T}}{(\log{X})^{m+1}}$, we have
\begin{align*}
\frac{1}{T}\meas([T, 2T] \setminus \mathcal{B})
\leq \exp\l( -b_{4} (\Delta(\log{T})^{m})^{1/(m+1)} \r).
\end{align*}
Here, $b_3$, $b_4$ are absolute positive constants.
\end{lemma}

\begin{proof}
By equation \eqref{RLEE} and Theorem 5 of \cite{II2019}, we have
\begin{align*}
&\frac{1}{T}\int_{T}^{2T}\l|\et_{m}(\s + it) 
- \sum_{2 \leq n \leq X}\frac{\Lam(n)}{n^{\s+it}(\log{n})^{m+1}}\r|^{2k}dt\\
&\ll C^k k! \frac{X^{k(1-2\s)}}{(\log{X})^{2km}} 
+ C^{k} k^{2k(m+1)}\frac{T^{\frac{1-2\s}{135}}}{(\log{T})^{2km}}
\end{align*}
for $3 \leq X \leq T^{\frac{1}{135k}}$, where $C$ is an absolute positive constant.
Therefore, we obtain
\begin{align*}
\frac{1}{T}\meas([T, 2T] \setminus \mathcal{B})
\ll \l( \frac{C k^{1/2}}{\Delta(\log{X})^{m}} \r)^{2k} 
+ \l( \frac{C k^{m+1}}{\Delta (\log{T})^{m}} \r)^{2k}.
\end{align*}
When $\Delta \leq \l(\frac{\log{T}}{(\log{X})^{2(m+1)}}\r)^{\frac{m}{2m+1}}$, 
putting $k = [c \Delta^2 (\log{X})^{2m}] + 1$ with $c$ a suitably small constant, we have
\begin{align*}
\frac{1}{T}\meas([T, 2T] \setminus \mc{B})
\leq \exp\l( -b_{3} \Delta^{2} (\log{X})^{2m} \r)
\end{align*}
for some absolute constant $b_{3} > 0$.
When the inequality $ \l(\frac{\log{T}}{(\log{X})^{2(m+1)}}\r)^{\frac{m}{2m+1}}
\leq \Delta \leq \frac{\log{T}}{(\log{X})^{m+1}}$ holds,
by choosing $k = \l[c(\Delta(\log{T})^{m})^{\frac{1}{m+1}}\r] + 1$ with $c$ a suitably small constant, we have
\begin{align*}
\frac{1}{T}\meas([T, 2T] \setminus \mc{B})
\leq \exp\l( -b_{4} (\Delta(\log{T})^{m})^{1/(m+1)} \r)
\end{align*}
for some absolute constant $b_{4} > 0$.
Thus, we obtain this lemma.
\end{proof}

\section{\textbf{Proofs of Proposition \ref{Main_Prop_EVE} and Theorem \ref{Main_Thm_EVE}}}

In this section, we prove Proposition \ref{Main_Prop_EVE} and Theorem \ref{Main_Thm_EVE}.

\begin{proof}[Proof of Proposition \ref{Main_Prop_EVE}]
Let $m \in \ZZ_{\geq 1}$, $\theta \in \RR$ be fixed.
Let $T$, $V$ be large numbers with $V \leq a_{2}\frac{\sqrt{\log{T}}}{(\log{\log{T}})^{m+\frac{1}{2}}}$, 
and let $X$ be a real parameter with $V^{4} \leq X \leq T^{a_{3} / V^{2} (\log{V})^{2m}}$.
Here, $a_{2} = a_{2}(m)$, $a_{3} = a_{3}(m)$ are positive constants to be chosen later.
Moreover, let $W > 0$, $3 \leq x \leq b_{2} W(\log{W})^{2m}$ be numbers to be chosen later,
where $b_{2} = b_{2}(1/2, m)$ is the same constant as in Lemma \ref{GRKL}.
Put
\begin{align*}
\S^{*}(T, V) 
:= \set{t \in \mathcal{A}}{\Re\bigg( e^{-i\theta} \sum_{p \leq X}\frac{1}{p^{1/2+it}(\log{p})^{m}}\bigg) > V}.
\end{align*}
Here, the set $\mathcal{A} = \mathcal{A}(T, X, W; 1/2, m)$ is defined by \eqref{def_sA}.
Then, for $x > 0$, we have
\begin{align*}
\int_{A}\exp\l( x\Re\bigg( e^{-i\theta} \sum_{p \leq X}\frac{1}{p^{1/2+it}(\log{p})^{m}}\bigg) \r)dt
= x\int_{-\infty}^{\infty}e^{x v}\meas(\S^{*}(T, v))dv.
\end{align*}
By this equation and Lemma \ref{GRKL}, it holds that
\begin{align*}
\frac{1}{T}\int_{-\infty}^{\infty}e^{x v}\meas(\S^{*}(T, v))dv
= \frac{1}{x}\prod_{p \leq X}I_{0}\l( \frac{x}{p^{\s}(\log{p})^{m}} \r) 
+ O\l( \frac{1}{x}\exp\l( - x W \r) \r)
\end{align*}
when $x^{3} \leq X \leq T^{1 / W^2(\log{W})^{2m}}$.
Therefore, by Lemma \ref{PB0E}, we obtain
\begin{align}	\label{PEVE1}
&\frac{1}{T}\int_{-\infty}^{\infty}e^{x v}\meas(\S^{*}(T, v))dv\\
&= \exp\l( \frac{x^2}{8m(2\log{x})^{2m}}\l( 1 - \l(\frac{\log{x^2}}{\log{X}}\r)^{2m} 
+ O_{m}\l( \frac{\log{\log{x}}}{\log{x}} \r) \r) \r)
\end{align}
for $x^{3} \leq X \leq T^{1 / W^2(\log{W})^{2m}}$.
Now, we decide the parameters $x$, $W$ as satisfying the equations
\begin{align}	\label{def_x_EVE}
V = \frac{2x}{8m (2\log{x})^{2m}}\l( 1 - \l( \frac{\log{x^2}}{\log{X}} \r)^{2m} \r),
\end{align}
and $W = 8m 4^{m} K_{1} V$, respectively. 
The constant $K_{1} = K_{1}(m)$ is defined as $K_{1} = \max\{b_{1}^{-1}, b_{2}^{-1}\}$, 
and $b_{1}$ is the same constant as in Lemma \ref{ESAE}.
Then, this $x$ satisfies 
\begin{align*}
x = \frac{4m 4^{m}}{1 - (\log{V^{2}} / \log{X})^{2m}} V(\log{V})^{2m}
\l( 1 + O_{m}(\log{\log{V}} / \log{V}) \r),
\end{align*} 
and hence we can take out $x$ from the range $3 \leq x \leq b_{2} W (\log{W})^{2m}$ for any large $V$.
Also, when $a_{2}$, $a_{3}$ are suitably small, 
the inequalities $x^{3} \leq T^{1 / W^2(\log{W})^{2m}}$ and 
$x^{3} \leq X \leq T^{1 / W^2(\log{W})^{2m}}$ hold for any large $V$.
Moreover, by using Lemma \ref{ESAE}, 
the inequality $\meas([T, 2T] \setminus \mathcal{A}) \leq T \exp\l( -8m 4^{m}V^{2}(\log{V})^{2m} \r)$ holds.
Therefore, we obtain
\begin{align}
&\frac{1}{T}\meas\set{t \in [T, 2T]}{\Re\l( e^{-i\theta}\sum_{p \leq X}\frac{1}{p^{1/2+it}(\log{p})^{m}} \r) > V}\\
&= \frac{1}{T}\meas(\S^{*}(T, V)) + O\l( \frac{1}{T}\meas([T, 2T] \setminus \mathcal{A}) \r)\\
\label{PEVE2}
&= \frac{1}{T}\meas(\S^{*}(T, V)) + O\l( \exp\l( -8m 4^{m}V^{2}(\log{V})^{2m} \r) \r).
\end{align}
Put 
$
\e = K_{2}\sqrt{\log{\log{x}} / \log{x}}
$
with $K_{2} = K_{2}(m)$ a sufficiently large constant.
Then, by using equation \eqref{PEVE1}, we find that 
\begin{align*}
&\int_{-\infty}^{V(1-\e)}e^{x v}\meas(\S^{*}(T, v))dv
\leq e^{\e x V(1-\e)}\int_{-\infty}^{\infty}e^{x(1-\e)v}\meas(\S^{*}(T, v))dv\\
&= T\exp\l( \frac{x^2}{8m(2\log{x})^{2m}}\l( 1 - \l( \frac{\log{x^2}}{\log{X}} \r)^{2m}
- \frac{\e^2}{3} + O_{m}\l( \frac{\log{\log{x}}}{\log{x}} \r) \r) \r)\\
&\leq \frac{1}{3}\int_{-\infty}^{\infty}e^{x v}\meas(\S^{*}(T, v))dv.
\end{align*}
Similarly, we find that
\begin{align*}
&\int_{V(1+\e)}^{\infty}e^{x v}\meas(\S^{*}(T, v))dv
\leq e^{-\e x V(1+\e)}\int_{-\infty}^{\infty}e^{x(1+\e)v}\meas(\S^{*}(T, v))dv\\
&= T\exp\l( \frac{x^2}{8m(2\log{x})^{2m}}\l( 1 - \l(\frac{\log{x^2}}{\log{X}} \r)^{2m} - \frac{\e^2}{3} 
+ O_{m}\l(\frac{\log{\log{x}}}{\log{x}} \r) \r) \r)\\&
\leq \frac{1}{3}\int_{-\infty}^{\infty}e^{x v}\meas(\S^{*}(T, v))dv.
\end{align*}
Hence, we have
\begin{align*}
&\frac{1}{T}\int_{V(1 - \e)}^{V(1 + \e)}e^{x v}\meas(\S^{*}(T, v))dv\\
&= \exp\l( \frac{x^2}{8m(2\log{x})^{2m}}\l( 1 - \l(\frac{\log{x^2}}{\log{X}}\r)^{2m} 
+ O_{m}\l( \frac{\log{\log{x}}}{\log{x}} \r) \r) \r).
\end{align*}
Moreover, since $\meas(\S^{*}(T, v))$ is a nonincreasing function 
and $\int_{V(1 - \e)}^{V(1 + \e)}e^{x v}dv = \exp(x V (1 + O(\e)))$, it holds that
\begin{align*}
&\frac{1}{T}\meas(\S^{*}(T, V(1 + \e)))\\
&\leq \exp\l( -\frac{x^2}{8m(2\log{x})^{2m}}\l( 1 - \l(\frac{\log{x^2}}{\log{X}}\r)^{2m} 
+ O_{m}\l( \sqrt{\frac{\log{\log{x}}}{\log{x}}} \r) \r) \r)\\
&\leq \frac{1}{T}\meas(\S^{*}(T, V(1 - \e))).
\end{align*}
In particular, since $x$ satisfies 
\begin{align*}
x = 4m V(2\log{V})^{2m}\l\{ \l(1 + (\log{x^2} / \log{X})^{2m}\r)^{-1}
+ O_{m}(\log{\log{V}} / \log{V}) \r\},
\end{align*} 
the second term of the above inequalities is equal to
\begin{align*}
\exp\l( - \frac{2m 4^{m}}{1 - \l(\frac{\log{V^2}}{\log{X}}\r)^{m}} V^2(\log{V})^{2m}\l(
1 + O_{m}\l( \sqrt{\frac{\log{\log{V}}}{\log{V}}} \r) \r) \r).
\end{align*}
Additionally, if we change the above $V$ to $V(1 + O(\e))$, the above form does not change.
Hence, we obtain
\begin{align*}
&\frac{1}{T}\meas(\S^{*}(T, V))
=\exp\l( - \frac{2m 4^{m} V^2(\log{V})^{2m}}{1 - \l(\frac{\log{V^2}}{\log{X}}\r)^{m}}
\l( 1 + O_{m}\l( \sqrt{\frac{\log{\log{V}}}{\log{V}}} \r) \r) \r).
\end{align*} 
By this equation and \eqref{PEVE2}, we complete the proof of Proposition \ref{Main_Prop_EVE}.
\end{proof}

\begin{proof}[Proof of Theorem \ref{Main_Thm_EVE}]
Let $T$, $V$ be sufficiently large parameters satisfying
$V \leq a_{1}\d{\l(\frac{\log{T}}{(\log{\log{T}})^{2m+2}}\r)^{\frac{m}{2m+1}}}$, 
where $a_{1} = a_{1}(m)$ is a suitably small constant to be chosen later.
Let $a_{3}$, $b_{4}$ be the same constants as in Proposition \ref{Main_Prop_EVE} and Lemma \ref{ESEPE}.
Put $X = T^{b_{5}/V^{2}(\log{V})^{2m}}$ with $b_{5} = \min\{ a_{3}, b_{4}(4m 4^{m})^{-1} \}$.
Note that this $X$ satisfies the inequality 
$X \geq \exp\l( (\log{T})^{\frac{1}{2m + 1} - \e} \r) \geq V^{4}$ when $T$ is large.
Then, applying Lemma \ref{ESEPE} as 
$\Delta = \frac{\log{T}}{(\log{X})^{m + 1}} = \frac{V^{2m+2}(\log{V})^{2m(m + 1)}}{b_{5}^{m+1}(\log{T})^{m}}$, 
we find that there exists a set $\mc{B} \subset [T, 2T]$ such that 
$\meas([T, 2T] \setminus \mc{B}) \leq T \exp\l(-4m 4^{m} V^2 (\log{V})^{2m}\r)$, 
and for all $t \in \mathcal{B}$
\begin{align*}
\l|\et_{m}(1/2 + it) - \sum_{p \leq X}\frac{1}{p^{1/2+it}(\log{p})^{m}}\r| 
&\leq \l(\frac{V^{2m+1}(\log{V})^{2m(m + 1)}}{b_{6}^{m+1}(\log{T})^{m}} + \frac{c}{V}\r)V\\
&=: \delta_{m} V,
\end{align*}
say. Here the constant $c$ indicates the value $\sum_{p^{k}, k \geq 2}\frac{1}{p^{k/2}(\log{p^{k}})^{m}}$. 
Now, we decide the number $a_{1}$ such that $\delta_{m} \leq 1/2$.
Then, it holds that
\begin{align*}
&\meas\set{t \in \mc{B}}{\Re e^{-i\theta} \sum_{p \leq X}\frac{1}{p^{1/2+it}(\log{p})^{m}} 
> V(1 + \delta_{m})}\\
&\leq \meas\set{t \in \mc{B}}{\Re e^{-i\theta} \et_{m}(1/2+it) > V}\\
&\leq \meas\set{t \in \mc{B}}{\Re e^{-i\theta} \sum_{p \leq X}\frac{1}{p^{1/2+it}(\log{p})^{m}} 
> V(1 - \delta_{m})}.
\end{align*}
Hence, by these inequalities and Proposition \ref{Main_Prop_EVE}, we have
\begin{align*}
&\frac{1}{T}\meas\set{t \in \mc{B}}{\Re e^{-i\theta} \et_{m}(1/2+it) > V} =\\
&\exp\l( - 2m 4^{m} V^2(\log{V})^{2m}
\l( 1 + O_{m}\l(\frac{V^{2m+1} (\log{V})^{2m(m+1)}}{(\log{T})^{m}} 
+ \sqrt{\frac{\log{\log{V}}}{\log{V}}} \r) \r) \r).
\end{align*}
Thus, by this equation and $\meas([T, 2T] \setminus \mc{B}) \leq T\exp\l(-4m4^{m} V^2(\log{V})^{2m} \r)$, 
we complete the proof of Theorem \ref{Main_Thm_EVE}.
\end{proof}

\section{\textbf{Proofs of Proposition \ref{Main_Prop_EVZE} and Theorem \ref{Main_Thm_EVZE}}}

Some parts in the proof of Proposition \ref{Main_Prop_EVZE} are written briefly because 
many points are similar to the proof of Proposition \ref{Main_Thm_EVE}.

\begin{proof}[Proof of Proposition \ref{Main_Prop_EVZE}]
Let $m \in \ZZ_{\geq 0}$, $\frac{1}{2} < \s < 1$ be fixed.
Let $T$, $V$ be large numbers with $V \leq a_{5}\frac{(\log{T})^{1-\s}}{(\log{\log{T}})^{m+1}}$,
and let $X$ be a real parameter with 
$V^{\frac{4}{1-\s}} \leq X \leq T^{a_{6} / V^{\frac{1}{1-\s}}(\log{V})^{\frac{m+\s}{1-\s}}}$.
Here $a_{5} = a_{5}(\s, m)$, $a_{6} = a_{6}(\s, m)$ are positive constants to be chosen later.
Moreover, let $W > 0$, $3 \leq x \leq b_{2}W^{\frac{\s}{1-\s}}(\log{W})^{\frac{m+\s}{1-\s}}$ 
be numbers to be chosen later. Here, $b_{2} = b_{2}(\s, m)$ is the same constant as in Lemma \ref{GRKL}.
Put
\begin{align*}
\S_{\s}^{*}(T, V)
:= \set{t \in \mathcal{A}}{\Re\l( e^{-i\theta}\sum_{p \leq X}\frac{1}{p^{\s+it}(\log{p})^{m}} \r) > V},
\end{align*}
where $\mathcal{A} = \mathcal{A}(T, X, V; \s, m)$ is the set defined by \eqref{def_sA}.
Using Lemmas \ref{PB0EG}, \ref{GRKL}, and the equation
\begin{align*}
\int_{\mathcal{A}}\exp\l( x \Re\l( e^{-i\theta} \sum_{p \leq X}\frac{1}{p^{\s}(\log{p})^{m}} \r) \r)
= x \int_{-\infty}^{\infty}e^{x v}\meas(\S_{\s}^{*}(T, v))dv,
\end{align*}
we obtain
\begin{align}	\label{MPEVZE1}
&\frac{1}{T}\int_{-\infty}^{\infty}e^{x v}\meas(\S_{\s}^{*}(T, v))dv\\
&= \exp\l( \frac{\s^{\frac{m}{\s}} G(\s) x^{\frac{1}{\s}}}{(\log{x})^{\frac{m}{\s} + 1}} 
\l(1 + O\l(\frac{1 + m\log{\log{x}}}{\log{x}}\r)\r) \r)
\end{align}
for $x^{3} \leq X \leq T^{1/W^{\frac{1}{1-\s}}(\log{W})^{\frac{m+\s}{1-\s}}}$.
Here, we decide the parameters $x$, $W$ as the numbers satisfying the equations
\begin{align}	\label{MPEVZE2}
V = \frac{\s^{\frac{m}{\s}} G(\s) x^{\frac{1}{\s} - 1}}{\s (\log{x})^{\frac{m}{\s} + 1}},
\end{align}
and $W = \l( 2\frac{A_{m}(\s)}{1-\s}K_{3} \r)^{\frac{1-\s}{\s}}V$, respectively. 
The constant $K_{3} = K_{3}(\s, m)$ is defined as $K_{3} = \max\{ b_{1}^{-1}, b_{2}^{-1} \}$, 
where $b_{1}$ is the same constant as in Lemma \ref{ESAE}. 
Then, this $x$ satisfies 
$
x 
= \frac{A_{m}(\s)}{1-\s}V^{\frac{\s}{1-\s}} (\log{V})^{\frac{m + \s}{1 - \s}}
(1 + O(\log{\log{V}} / \log{V}))
$
, and so we can pick up this $x$ 
from the range $3 \leq x \leq b_{2} W^{\frac{\s}{1-\s}} (\log{W})^{\frac{m + \s}{1 - \s}}$
for any large $V$.
Also, choosing $a_{5}$, $a_{6}$ as suitably small constants, 
we find that the inequalities $x^{3} \leq T^{1 / W^{\frac{1}{1-\s}}(\log{W})^{\frac{m + \s}{1-\s}}}$ and
$x^{3} \leq X \leq T^{1 / W^{\frac{1}{1-\s}}(\log{W})^{\frac{m + \s}{1-\s}}}$ hold for any large $V$.
Moreover, by Lemma \ref{ESAE}, the inequality 
$\meas([T, 2T] \setminus \mathcal{A}) \leq T\exp\l( -2A_{m}(\s)V^{\frac{1}{1-\s}}(\log{V})^{\frac{m+\s}{1-\s}} \r)$
holds.

Putting $\e = K _{4}\sqrt{\frac{1 + m\log{\log{x}}}{\log{x}}}$ with $K_{4} = K_{4}(\s, m)$ a suitably large constant 
and using equation \eqref{MPEVZE1}, we have
\begin{align*}
&\int_{-\infty}^{V(1-\e)}e^{x v}\meas(\S_{\s}^{*}(T, v; X))dv
\leq \frac{1}{3}\int_{-\infty}^{\infty}e^{x v}\meas(\S_{\s}^{*}(T, v; X))dv,
\end{align*}
and
\begin{align*}
&\int_{V(1+\e)}^{\infty}e^{x v}\meas(\S_{\s}^{*}(T, v; X))dv
\leq \frac{1}{3}\int_{-\infty}^{\infty}e^{x v}\meas(\S_{\s}^{*}(T, v; X))dv.
\end{align*}
Therefore, we obtain
\begin{align*}
&\frac{1}{T}\int_{(1 - \e)V}^{(1 + \e)V}e^{x v}\meas(\S_{\s}^{*}(T, v))dv\\
&= \exp\l( \frac{\s^{\frac{m}{\s}} G(\s) x^{\frac{1}{\s}}}{(\log{x})^{\frac{m}{\s} + 1}} 
\l(1 + O\l(\frac{1 + m\log{\log{x}}}{\log{x}}\r)\r) \r).
\end{align*}
Moreover, since $\meas(\S^{*}(T, v))$ is a nonincreasing function 
and $\int_{V(1 - \e)}^{V(1 + \e)}e^{x v}dv = \exp(x V (1 + O(\e)))$, it holds that
\begin{align*}
&\frac{1}{T}\meas(\S^{*}(T, V(1 + \e); X))\\
&\leq \exp\l( -\frac{1-\s}{\s}\frac{\s^{\frac{m}{\s}}G(\s)x^{\frac{1}{\s}}}{(\log{x})^{\frac{m}{\s}+1}}
\l( 1 + O\l( \e \r) \r) \r)\\
&\leq \frac{1}{T}\meas(\S^{*}(T, V(1 - \e); X)).
\end{align*}
In particular, as $x$ is the solution of equation \eqref{MPEVZE1},  
the above second term is equal to
\begin{align*}
\exp\l( - A_{m}(\s) V^{\frac{1}{1-\s}}(\log{V})^{\frac{m+\s}{1 - \s}}\l( 1 + R \r) \r), 
\end{align*}
where 
\begin{align*}
R 
&\ll \sqrt{\frac{1 + m\log{\log{x}}}{\log{x}}}
\ll \sqrt{\frac{1 + m\log{\log{V}}}{\log{V}}}.
\end{align*}
Additionally, if we change the above $V$ to $V(1 + O(\e))$, the above form does not change.
Thus, we obtain
\begin{align*}
&\frac{1}{T}\meas(\S^{*}(T, V; X))\\
&=\exp\l( - A_{m}(\s) V^{\frac{1}{1-\s}}(\log{V})^{\frac{m+\s}{1 - \s}}
\l( 1 + O\l(\sqrt{\frac{1 + m\log{\log{V}}}{\log{V}}}\r) \r) \r).
\end{align*} 
By this equation 
and $\meas([T, 2T] \setminus \mathcal{A}) \leq T\exp\l( -2A_{m}(\s)V^{\frac{1}{1-\s}}(\log{V})^{\frac{m+\s}{1-\s}} \r)$, 
we obtain Proposition \ref{Main_Prop_EVZE}.
\end{proof}

\begin{proof}[Proof of Theorem \ref{Main_Thm_EVZE}]
We show only the case $m \geq 1$ because
the case $m = 0$ can be shown similarly by use of Lemma 2.2 in \cite{GS2006} instead of Lemma \ref{ESEPE}.

Let $m \in \ZZ_{\geq 1}$, $1/2 < \s < 1$.
Let $a_{5}$, $a_{6}$, and $b_{4}$ be the same constants as in Proposition \ref{Main_Prop_EVZE} and Lemma \ref{ESEPE}.
Let $T$, $V$ be sufficiently large positive numbers satisfying 
the inequality $V \leq a_{4}\frac{(\log{T})^{1-\s}}{(\log{\log{T}})^{m+1}}$, 
where $a_{4} = a_{4}(\s, m)$ is a suitably small  constant less than $a_{5}$ to be chosen later.
Put $X = T^{b_{6}/ V^{\frac{1}{1-\s}} (\log{V})^{\frac{m+\s}{1-\s}}}$ with 
$b_{6} = \min\{ a_{6}, b_{4} (2A_{m}(\s))^{-1} \}$.
Then we decide the number $a_{4}$ as satisfying
$X^{\s - 1/2} \geq (\log{T})^{6}$.
Applying Lemma \ref{ESEPE} as $\Delta = \frac{\log{T}}{(\log{X})^{m+1}} 
= \frac{\l(V^{\frac{1}{1-\s}} (\log{V})^{\frac{m+\s}{1-\s}}\r)^{m+1}}{b_{6}^{m+1}(\log{T})^{m}}$, 
we find that there exists a set $\mathcal{B} \subset [T, 2T]$ such that
$\meas([T, 2T] \setminus \mathcal{B}) \leq T\exp\l(-2A_{m}(\s) V^{\frac{1}{1-\s}} (\log{V})^{\frac{m+\s}{1-\s}} \r)$, 
and for all $t \in \mathcal{B}$
\begin{align*}
\l|\et_{m}(\s + it) - \sum_{p \leq X}\frac{1}{p^{\s+it}(\log{p})^{m}}\r| 
&\leq \frac{\l(V^{\frac{1}{1-\s}} (\log{V})^{\frac{m+\s}{1-\s}}\r)^{m+1}}{X^{\s-1/2}b_{6}^{m+1}(\log{T})^{m}} + c.
\end{align*}
Here, $c = \sum_{p^{k}, k\geq 2}\frac{\Lam(p^{k})}{p^{k\s}(\log{p^{k}})^{m+1}}$. 
Therefore, the right hand side is $\leq K_{4}$ with $K_{4} = K_{4}(m, \s)$ a positive constant.
Then, it holds that
\begin{align*}
&\meas\set{t \in \mathcal{B}}{\Re e^{-i\theta} \sum_{p \leq X}\frac{1}{p^{\s+it}(\log{p})^{m}} 
> V(1 + K_{4} V^{-1})}\\
&\leq \meas\set{t \in \mathcal{B}}{\Re e^{-i\theta} \et_{m}(\s+it) > V}\\
&\leq \meas\set{t \in \mathcal{B}}{\Re e^{-i\theta} \sum_{p \leq X}\frac{1}{p^{\s+it}(\log{p})^{m}} 
> V(1 - K_{4}V^{-1})}.
\end{align*}
Hence, by these inequalities and Proposition \ref{Main_Prop_EVE}, we have
\begin{align*}
&\frac{1}{T}\meas\set{t \in \mathcal{B}}{\Re e^{-i\theta} \et_{m}(\s+it) > V}\\
&= \exp\l( - A_{m}(\s) V^{\frac{1}{1-\s}}(\log{V})^{\frac{m+\s}{1-\s}}
\l( 1 + O\l(\sqrt{\frac{1 + m\log{\log{V}}}{\log{V}}}\r) \r) \r).
\end{align*}
By this equation and 
$\meas([T, 2T] \setminus \mathcal{B}) \leq T\exp\l(-2 A_{m}(\s) V^{\frac{1}{1-\s}}(\log{V})^{\frac{m+\s}{1-\s}} \r)$, 
we complete the proof of Theorem \ref{Main_Thm_EVZE}.
\end{proof}

\begin{acknowledgment*}
The author would like to thank Professor Kohji Matsumoto for his helpful comments.
This work is supported by Grant-in-Aid for JSPS Research Fellow (Grant Number: 19J11223).
\end{acknowledgment*}


\begin{thebibliography}{99}








\bibitem{BJ1930} H. Bohr and B. Jessen, \"{U}ber die Werteverteilung der Riemannschen Zetafunktion, Erste Mitteilung, 
\textit{Acta Math.} \textbf{54} (1930), 1--35; Zweite Mitteilung, ibid. \textbf{58} (1932), 1--55.




\bibitem{BS2018} A. Bondarenko and K. Seip, Extreme values of the Riemann zeta function and its argument, 
\textit{Math. Ann.} \textbf{372} (2018), 999--1015.


\bibitem{BS2017} A. Bondarenko and K. Seip, 
Large greatest common divisor sums and extreme values of the Riemann zeta function, 
\textit{Duke Math. J.} \textbf{166} no.9 (2017) 1685--1701.






\bibitem{Ch2019} A. Chirre, Extreme values for $S_{n}(\s, t)$ near the critical line, 
\textit{J. Number Theory} \textbf{200} (2019), 329--352.


\bibitem{CM2020} A. Chirre and K. Mahatab, Large values of the argument of the Riemann zeta-function and its iterates, 
preprint, \texttt{arXiv:2006.04288}.












\bibitem{Fu2002} A. Fujii, On the zeros of the Riemann zeta-function, 
\textit{Comment. Math. Univ. St. Paul.} \textbf{51} (2002), no.1, 1--17.




\bibitem{GS2006} A. Granville and K. Soundararajan, Extreme values of $|\zeta(1 + it)|$, 
\textit{The Riemann Zeta Function and Related Themes: Papers in Honor of Professor K. Ramachandra}, 
Ramanujan Math. Soc. Lecture Notes Series 2, Mysore, 2006, pp.65--80.






\bibitem{GHK2007} S. M. Gonek, C. P. Hughes, and J. P. Keating, 
A hybrid Euler-Hadamard product for the Riemann zeta-function, 
\textit{Duke Math. J.} \textbf{136} no.3 (2007), 507--549.






\bibitem{HM1999} T. Hattori and K. Matsumoto, 
A limit theorem for Bohr-Jessen's probability measures of the Riemann zeta-function, 
\textit{J. Reine Angew. Math.} \textbf{507} (1999), 219--232.








\bibitem{II2019} S. Inoue, On the logarithm of the Riemann zeta-function and its iterated integrals, 
preprint, \texttt{arXiv:1909.03643}.












\bibitem{L2011} Y. Lamzouri, Distribution of large values of zeta and $L$-functions, 
\textit{Int. Math. Res. Not. IMRN} \textbf{2011} no.23 (2011), 5449--5503.


\bibitem{L1924} J. E. Littlewood, On the zeros of the Riemann zeta-function,
\textit{Proc. Camb. Phil. Soc.} \textbf{22} (1924), 295--318.






\bibitem{Ra2011} M. Radziwi\l\l, Large deviations in Selberg's central limit theorem, \texttt{arXiv:1108.5092}.











\bibitem{SS} A. Selberg, On the remainder formula for $N(T)$,
the Number of Zeros of $\zeta(s)$ in the Strip $0 < t < T$, 
Avhandl. Norske Vid.-Akad. Olso I. Mat.-Naturv. Kl., no.1;
Collected Papers, Vol. 1, New York: Springer Verlag. 1989, 179--203.


\bibitem{SCR} A. Selberg, Contributions to the theory of the Riemann zeta-function, 
Avhandl. Norske Vid.-Akad. Olso I. Mat.-Naturv. Kl., no.1;
Collected Papers, Vol. 1, New York: Springer Verlag. 1989, 214--280.




\bibitem{SE2008} K. Soundararajan, Extreme values of zeta and $L$-functions, 
\textit{Math.\ Ann.} \textbf{342} (2008), 467--486.




\bibitem{SM2009} K. Soundararajan, Moments of the Riemann zeta-function, 
\textit{Ann. of Math.} (2) \textbf{170} (2009), 981--993.




\bibitem{T}  E. C. Titchmarsh, \textit{The Theory of the Riemann Zeta-Function}, 
Second Edition, Edited and with a preface by D. R. Heath-Brown, 
The Clarendon Press, Oxford University Press, New York, 1986.  


\bibitem{Ts1993} K. M. Tsang, The large values of the Riemann zeta-function, \textit{Mathematika} \textbf{40}
(1993), 203--214.


\bibitem{Ts1986} K. M. Tsang, Some $\Omega$-theorems for the Riemann zeta-function, \textit{Acta Arith.} 
\textbf{46} (1986), no.4, 369--395.








\end{thebibliography}
\end{document}